\documentclass[a4paper]{article}

\usepackage{amsmath,amssymb,amsthm}
\usepackage[left=2.9cm, right=2.9cm, top=2.9cm, bottom=2.9cm]{geometry}
\usepackage{hyperref}
\usepackage{mathtools}
\usepackage{pgfplots}
\usepackage{subfig}
\usepackage{tikz}
\usetikzlibrary{quotes,angles}
\usepackage{xcolor}

\allowdisplaybreaks

\pgfplotsset{trig format plots=rad, compat=1.16}

\newtheorem{theorem}{Theorem}[section]
\newtheorem{lemma}[theorem]{Lemma}
\newtheorem{remark}[theorem]{Remark}

\newtheorem{question}[theorem]{Question}

\newcommand{\calF}{\mathcal{F}}

\newcommand{\calV}{\mathcal{V}}

\newcommand{\bbC}{\mathbb{C}}

\newcommand{\bbR}{\mathbb{R}}
\newcommand{\bbZ}{\mathbb{Z}}

\newcommand{\rme}{\mathrm{e}}
\newcommand{\rmi}{\mathrm{i}}

\newcommand{\abs}[1]{\left\lvert #1 \right\rvert}

\newcommand{\dd}{\,\mathrm{d}}

\begin{document}

\title{Phase retrieval from sampled Gabor transform magnitudes: Counterexamples}
\author{Rima Alaifari\thanks{ETH Zurich, Seminar for Applied Mathematics, HG G 58.3,
Ramistrasse 101, 8092 Zurich, Switzerland} \and Matthias Wellershoff\thanks{ETH Zurich, Seminar for
Applied Mathematics, HG G 58.3, Ramistrasse 101, 8092 Zurich, Switzerland,
\href{mailto:matthias.wellershoff@sam.math.ethz.ch}{\texttt{matthias.wellershoff@sam.math.ethz.ch}}.}}
\date{\today}

\maketitle

\begin{abstract}
    We consider the recovery of square-integrable signals from discrete, equidistant samples of their Gabor transform magnitude and show that, in general, signals can not be recovered from such samples. In particular, we show that for any lattice, one can construct functions in $L^2(\bbR)$ which do not agree up to global phase but whose Gabor transform magnitudes sampled on the lattice agree. These functions have good concentration in both time and frequency and can be constructed to be real-valued for rectangular lattices.

    \vspace{5pt}

    \vspace{5pt}
\end{abstract}

\section{Introduction}
\label{sec:introduction}

Let us consider the Gaussian $\phi(t) = \rme^{-\pi t^2}$, for $t \in \bbR$. We may define the Gabor transform of a signal $f \in L^2(\bbR)$ via 
\[
    \calV_\phi f(x,\omega) = \int_\bbR f(t) \phi(t-x) \rme^{-2\pi\rmi t \omega} \dd t,
    \qquad (x,\omega) \in \bbR^2.
\]
In this paper, we are interested in the uniqueness question of the \emph{Gabor phase retrieval} problem which consists of recovering a function $f \in L^2(\bbR)$ from the magnitude measurements
\begin{equation}
    \label{eq:measurements}
    \abs{\calV_\phi f(x,\omega)}, \qquad (x,\omega) \in S,
\end{equation}
where $S$ is a subset of $\bbR^2$. We say that $f$ is uniquely determined (up to a global phase factor) by the measurements \eqref{eq:measurements} if for any $g \in L^2(\mathbb{R})$, 
\[
    \abs{\calV_\phi f(x,\omega)} = \abs{\calV_\phi g(x,\omega)} \qquad (x,\omega) \in S,
\]
implies that 
\[
    f = \rme^{\rmi \mu} g,
\]
for some $\mu \in \mathbb{R}$. When $S=\mathbb{R}^2$, it is well-known that Gabor phase retrieval is uniquely solvable. In fact, this result holds true for any window function $\psi$ for which $\calV_\psi \psi$ is non-zero almost everywhere on $\bbR^2$. However, when $S$ is a true subset of $\mathbb{R}^2$, the answer is less clear. In particular, measurements can only be collected on discrete sets $S$ in applications. Thus, the question of uniqueness for Gabor phase retrieval is specifically interesting when $S$ is discrete.

In recent work \cite{alaifari2021uniqueness}, we were able to show that \emph{real-valued, bandlimited} signals in $L^2(\bbR)$ are uniquely determined up to global phase from Gabor magnitude measurements \eqref{eq:measurements} sampled on the discrete set $S = (4B)^{-1} \bbZ \times \{0\}$, where $B > 0$ is such that the bandwidth of the signal is contained in $[-B,B]$. While we were writing this paper, work by Grohs and Liehr \cite{grohs2020injectivity} appeared showing that it is possible to recover \emph{compactly supported} signals in $L^4([-C/2,C/2])$ up to global phase from Gabor magnitude measurements \eqref{eq:measurements} sampled on the discrete set $S = \bbZ \times (2C)^{-1} \bbZ$. Finally, during the review process of this paper, the work \cite{grohs2021foundational} appeared. In \cite{grohs2021foundational}, the authors generalise the findings proposed here to general short-time Fourier transform phase retrieval.

We focus on the general uniqueness question for Gabor phase retrieval:
\begin{question}
    \label{q:research_question}
    Is there any lattice $S \subset \bbR^2$ such that all functions in $L^2(\bbR)$ are uniquely determined up to global phase from Gabor magnitude measurements \eqref{eq:measurements} sampled on $S$?
\end{question}

The main contribution of this paper is that we answer this question negatively: In particular, no matter how fine-grained the sampling lattice $S$, one will not be able to recover all functions in $L^2(\bbR)$ from Gabor magnitude measurements \eqref{eq:measurements} on $S$. Our answer to Question \ref{q:research_question} is constructive in the sense that we are able to explicitly give functions $f_{\pm} \in L^2(\bbR)$ which do not agree up to global phase but which satisfy 
\begin{equation*}
    \abs{\calV_\phi f_+(x,\omega)} = \abs{\calV_\phi f_-(x,\omega)}, \qquad (x,\omega) \in S.
\end{equation*}
We note that the functions $f_{\pm} \in L^2(\bbR)$ are well concentrated in both time and frequency and if $S = a \bbZ \times b \bbZ$, $a,b > 0$, is a rectangular lattice, then $f_+$ and $f_-$ can be constructed to be real-valued.

\subsection{Basic notions}
\label{ssec:basic_notions}

We want to emphasise that the notation $\phi$ is reserved for the Gaussian $\phi(t) = \rme^{-\pi t^2}$, for $t \in \bbR$, throughout this paper.

We will encounter the \emph{fractional Fourier transform} of a function $f \in L^1(\bbR)$ defined by 
\begin{equation*}
    \calF_\alpha f(\omega) := c_\alpha \rme^{\pi \rmi \omega^2 \cot \alpha} \int_\bbR f(t) \rme^{\pi\rmi t^2 \cot \alpha} \rme^{-2 \pi \rmi \frac{t \omega}{\sin \alpha}} \dd t, \qquad \omega \in \bbR,
\end{equation*}
for $\alpha \in \bbR \setminus \pi \bbZ$, where $c_\alpha \in \bbC$ is the square root of $1 - \rmi \cot \alpha$ with positive real part, and by $\calF_{2\pi k} f := f$ as well as $\calF_{(2k+1)\pi} f(\omega) := f(-\omega)$, for $\omega \in \bbR$, where $k \in \bbZ$ \cite{jaming2014uniqueness}. One may show that the fractional Fourier transform preserves the canonical inner product $(\cdot,\cdot)$ on $L^2(\bbR)$ in the sense that for all $\alpha \in \bbR$ and $f,g \in L^1(\bbR) \cap L^2(\bbR)$ it holds that 
\begin{equation*}
    (f,g) = (\calF_\alpha f, \calF_\alpha g).
\end{equation*}
It follows that one may extend the fractional Fourier transform to $L^2(\bbR)$ by a classical density argument. For the further understanding of this paper, it is essential to observe that the fractional Fourier transform rotates functions in the time-frequency plane in the following sense: 

\begin{lemma}[See p.~424 of \cite{jaming2014uniqueness}]
    \label{lem:frftandgabor}
    Let $\alpha \in \bbR$ and $f,g \in L^2(\bbR)$. It holds that 
    \begin{equation*}
        \calV_{\calF_\alpha g} \calF_\alpha f (x,\omega) = \calV_g f (x \cos \alpha - \omega \sin \alpha, x \sin \alpha + \omega \cos \alpha) \rme^{\pi \rmi \sin \alpha \left( \left( x^2 - \omega^2 \right) \cos \alpha - 2 x\omega \sin \alpha \right)},
    \end{equation*}
    for $x,\omega \in \bbR$.
\end{lemma}

In addition, as for the classical Fourier transform, it holds that the Gaussian $\phi$ is invariant under the fractional Fourier transform in the sense that 
\begin{equation*}
    \calF_\alpha \phi = \phi, \qquad \alpha \in \bbR.
\end{equation*}
One can see this by a direct computation using the classical result which can for instance be found on p.~17 of \cite{grochenig2001foundations}.

In the following, we will denote by $R_\alpha : \bbR^2 \to \bbR^2$ the rotation of the time-frequency plane given by 
\begin{equation*}
    R_\alpha(x,\omega) := (x \cos \alpha - \omega \sin \alpha, x \sin \alpha + \omega \cos \alpha), \qquad x,\omega \in \bbR.
\end{equation*}
Additionally, we will use the standard notation for the family of time-shift operators $\operatorname{T}_{x} : L^2(\bbR) \to L^2(\bbR)$, with $x \in \bbR$, given by 
\begin{equation*}
    \operatorname{T}_{x} f(t) = f(t-x), \qquad t \in \bbR,
\end{equation*}
and the family of modulation operators $\operatorname{M}_{\xi} : L^2(\bbR) \to L^2(\bbR)$, with $\xi \in \bbR$, given by 
\begin{equation*}
    \operatorname{M}_{\xi} f(t) = f(t) \rme^{2\pi\rmi t \xi}, \qquad t \in \bbR,
\end{equation*}
for $f \in L^2(\bbR)$.

\section{Main result}

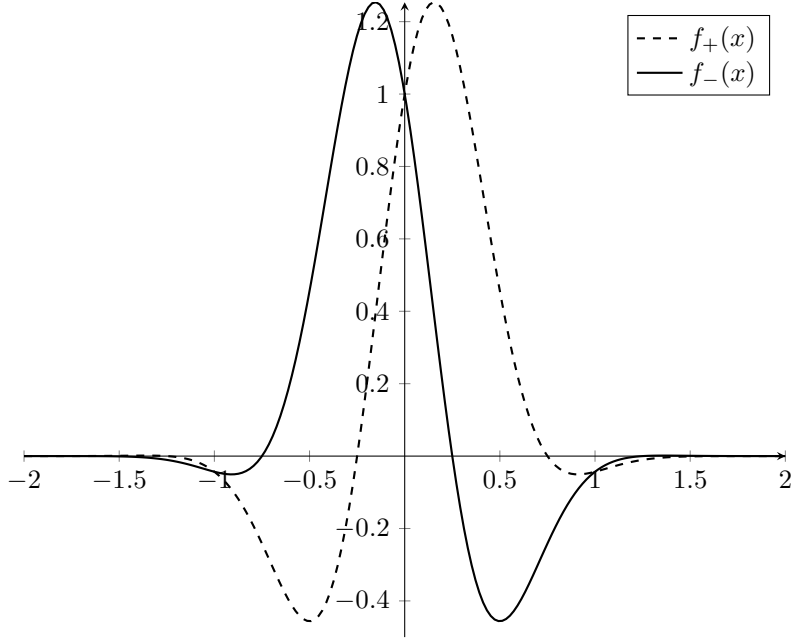
\begin{figure}
    \centering
    \begin{tikzpicture}
        \begin{axis}[ymin = -0.5, height=10cm, axis lines = middle]
            \addplot[domain=-2:2, samples=200, dashed, thick] {exp(-pi*x^2)*(cos(pi*x) + sin(pi*x))};
            \addlegendentry{$f_+(x)$}
            \addplot[domain=-2:2, samples=200, thick] {exp(-pi*x^2)*(cos(pi*x) - sin(pi*x))};
            \addlegendentry{$f_-(x)$}
        \end{axis}
    \end{tikzpicture}
    \caption{The functions defined in equation \eqref{eq:fpm} with $a=1$.}
    \label{fig:fpm}
\end{figure}

Using the relation of the Gabor transform and the Bargmann transform
\cite{grochenig2001foundations} in conjunction with the Hadamard factorisation theorem allows us to design functions $f_{\pm} \in L^2(\bbR)$ which do not agree up to global phase but which generate measurements \eqref{eq:measurements} that agree when $S \subset \bbR^2$ is chosen to be any set of infinitely many equidistant parallel lines. For the specific set $S = a\bbZ \times \bbR$, with $a > 0$, this construction leads us to consider the \emph{real-valued} functions 
\begin{equation}
    \label{eq:fpm}
    f_{\pm}(t) := \rme^{-\pi t^2} \left( \cos \left( \frac{\pi t}{a} \right) \pm \sin \left( \frac{\pi t}{a} \right) \right), \qquad t \in \bbR,
\end{equation}
as depicted in Figure \ref{fig:fpm}. It is easy to see that $f_+$ and $f_-$ do not agree up to global phase: Consider for instance that 
\begin{equation*}
    f_+\left( \frac{a}{4} \right) = \rme^{-\frac{\pi a^2}{16}} \left( \cos\left( \frac{\pi}{4} \right) + \sin\left( \frac{\pi}{4} \right) \right) = \sqrt{2} \rme^{-\frac{\pi a^2}{16}} \neq 0,
\end{equation*}
as well as 
\begin{equation*}
    f_-\left( \frac{a}{4} \right) = \rme^{-\frac{\pi a^2}{16}} \left( \cos\left( \frac{\pi}{4} \right) - \sin\left( \frac{\pi}{4} \right) \right) = 0.
\end{equation*}
Furthermore, it is not hard to compute the Gabor transforms of $f_+$ and $f_-$ and thereby verify the following lemma.

\begin{lemma}
    \label{lem:gaborfpm}
    Let $a > 0$ and let $f_\pm \in L^2(\bbR)$ be defined as in equation \eqref{eq:fpm}. Then, it holds that 
    \begin{equation}
        \label{eq:agree_on_parallel_lines}
        \abs{\calV_\phi f_+} = \abs{\calV_\phi f_-} \mbox{ on } a \bbZ \times \bbR.
    \end{equation}
\end{lemma}

\begin{remark}
    It is immediate that the above lemma implies that 
    \begin{equation*}
        \abs{\calV_\phi f_+} = \abs{\calV_\phi f_-}
    \end{equation*}
    continues to hold on $a \bbZ \times b \bbZ$ no matter how small one chooses $b > 0$.
\end{remark}

\begin{remark}
    As mentioned above, the functions $f_+$ and $f_-$ were constructed by considering what equation \eqref{eq:agree_on_parallel_lines} implies for the Bargmann transforms $F_+$ and $F_-$ of $f_+$ and $f_-$. In particular, we applied Hadamard's factorisation theorem to $F_+$ and $F_-$ and then followed ideas similar to the ones presented in \cite{jaming2014uniqueness,mc2004phase}.
\end{remark}

\begin{proof}[Proof of Lemma \ref{lem:gaborfpm}]
    Let us start by noting that 
    \begin{equation*}
        f_{\pm} = \left( \frac{1}{2} \pm \frac{1}{2\rmi} \right) \operatorname{M}_{\frac{1}{2a}} \phi + \left( \frac{1}{2} \mp \frac{1}{2\rmi} \right) \operatorname{M}_{-\frac{1}{2a}} \phi.
    \end{equation*}
    By the linearity and the covariance property of the Gabor transform (see Lemma 3.1.3 on p.~41 of \cite{grochenig2001foundations}), we obtain 
    \begin{equation*}
        \calV_\phi f_{\pm}(x,\omega) = \left( \frac{1}{2} \pm \frac{1}{2\rmi} \right) \calV_\phi \phi\left( x, \omega - \frac{1}{2a} \right) + \left( \frac{1}{2} \mp \frac{1}{2\rmi} \right) \calV_\phi \phi\left( x, \omega + \frac{1}{2a} \right),
    \end{equation*}
    for $x,\omega \in \bbR$. Using that the Gaussian is invariant under the Fourier transform, one may calculate that
    \begin{equation*}
        \calV_\phi \phi (x,\omega) = \frac{1}{\sqrt{2}} \rme^{-\pi \rmi x \omega} \rme^{-\frac{\pi}{2}\left(x^2 + \omega^2\right)}
    \end{equation*}
    such that
    \begin{equation*}
        \calV_\phi f_{\pm}(x,\omega) = \frac{1}{2 \sqrt{2}} ( 1 \mp \rmi ) \rme^{-\pi \rmi x \left( \omega - \frac{1}{2a} \right)} \rme^{-\frac{\pi}{2} \left( x^2 + \left( \omega-\frac{1}{2a} \right)^2 \right)} + \frac{1}{2 \sqrt{2}} ( 1 \pm \rmi ) \rme^{-\pi \rmi x \left( \omega + \frac{1}{2a} \right)} \rme^{-\frac{\pi}{2} \left( x^2 + \left( \omega+\frac{1}{2a} \right)^2 \right)}.
    \end{equation*}
    We might reformulate the above expression to 
    \begin{equation*}
        \calV_\phi f_{\pm}(x,\omega) = \frac{\rme^{-\frac{\pi}{8a^2}}}{2 \sqrt{2}} \rme^{-\pi\rmi x\omega} \rme^{-\frac{\pi}{2}\left(x^2 + \omega^2\right)} \left( ( 1 \mp \rmi ) \rme^{\frac{\pi}{2a}\left( \omega + \rmi x \right)} + ( 1 \pm \rmi ) \rme^{-\frac{\pi}{2a}\left( \omega + \rmi x \right)} \right).
    \end{equation*}
    If $x = ak$, where $k \in \bbZ$, and $\omega \in \bbR$, then it holds that 
    \begin{align*}
        ( 1 - \rmi ) \rme^{\frac{\pi}{2a}\left( \omega + \rmi x \right)} + ( 1 + \rmi ) \rme^{-\frac{\pi}{2a}\left( \omega + \rmi x \right)} &= ( 1 - \rmi ) \rmi^k \rme^{\frac{\pi \omega}{2a}} + ( 1 + \rmi ) (-\rmi)^k \rme^{-\frac{\pi \omega}{2a}} \\
        &= \overline{ ( 1 + \rmi ) (-\rmi)^k \rme^{\frac{\pi \omega}{2a}} + ( 1 - \rmi ) \rmi^k \rme^{-\frac{\pi \omega}{2a}} } \\
        &= (-1)^k \cdot \overline{ ( 1 + \rmi ) \rmi^k \rme^{\frac{\pi \omega}{2a}} + ( 1 - \rmi ) (-\rmi)^k \rme^{-\frac{\pi \omega}{2a}} } \\
        &= (-1)^k \cdot \overline{ ( 1 + \rmi ) \rme^{\frac{\pi}{2a}\left( \omega + \rmi x \right)} + ( 1 - \rmi ) \rme^{-\frac{\pi}{2a}\left( \omega + \rmi x \right)} }.
    \end{align*}
    It follows that 
    \begin{equation*}
        \abs{\calV_\phi f_+(x,\omega)} = \abs{\calV_\phi f_-(x,\omega)},
    \end{equation*}
    for $x \in a \bbZ$ and $\omega \in \bbR$. 
\end{proof}

\begin{remark}
    According to the preceding proof, it holds that 
    \begin{equation*}
        \calV_\phi f_{\pm}(x,\omega) = \frac{\rme^{-\frac{\pi}{8a^2}}}{\sqrt{2}} \rme^{-\pi\rmi x\omega} \rme^{-\frac{\pi}{2}\left(x^2 + \omega^2\right)} \left( \cosh\left( \frac{\pi}{2a} \left( \omega + \rmi x \right) \right) \mp \rmi \sinh\left( \frac{\pi}{2a} \left( \omega + \rmi x \right) \right) \right), \qquad x,\omega \in \bbR.
    \end{equation*}
\end{remark}

To generalise Lemma \ref{lem:gaborfpm}, we can use the fractional Fourier transform as well as time shifts and modulations to rotate and shift the functions $f_\pm$ in the time frequency plane. In this way, we construct the functions
\begin{equation}
    \label{eq:fpm_refined}
    f_{\pm}^{\alpha,\lambda} := \operatorname{T}_{x_0} \operatorname{M}_{\omega_0} \calF_{-\alpha} f_{\pm},
\end{equation}
for $\alpha \in \bbR$ and $\lambda = (x_0,\omega_0) \in \bbR^2$. We may now make use of well-known properties of the Gabor transform to show the following result:

\begin{theorem}[Main theorem]
    \label{thm:mainthm}
    Let $a > 0$, $\alpha \in \bbR$ and $\lambda = (x_0,\omega_0) \in \bbR^2$. Let furthermore $f_{\pm}^{\alpha,\lambda} \in L^2(\bbR)$ be defined as in equation \eqref{eq:fpm_refined}. Then, it holds that $f_+^{\alpha,\lambda}$ and $f_-^{\alpha,\lambda}$ do not agree up to global phase and yet
    \begin{equation*}
        \abs{\calV_\phi f_+^{\alpha,\lambda}} = \abs{\calV_\phi f_-^{\alpha,\lambda}} \mbox{ on } R_\alpha(a \bbZ \times \bbR) + \lambda.
    \end{equation*}
    Furthermore, it holds that for all $\epsilon > 0$, there exists a constant $C_\epsilon \geq 1$ (that additionally depends on $a$ and $\lambda$) such that 
    \begin{equation*}
        \abs{\calV_\phi f_{\pm}^{\alpha,\lambda}(x,\omega)} \leq C_\epsilon \rme^{-\left( \frac{\pi}{2}-\epsilon \right)\left( x^2 + \omega^2 \right)}, \qquad x,\omega \in \bbR.
    \end{equation*}
\end{theorem}

\begin{remark}
    It is immediate that the above theorem implies that 
    \begin{equation*}
        \abs{\calV_\phi f_+^{\alpha,\lambda}} = \abs{\calV_\phi f_-^{\alpha,\lambda}}
    \end{equation*}
    continues to hold on the lattice $S = R_\alpha(a \bbZ \times b\bbZ) + \lambda$ no matter how small one chooses $b > 0$.
\end{remark}

\begin{proof}[Proof of Theorem \ref{thm:mainthm}]
    We directly compute that 
    \begin{equation}
    \label{eq:central_to_proof}
    \begin{aligned}
        \abs{\calV_\phi f_{\pm}^{\alpha,\lambda}(x,\omega) } &= \abs{ \calV_\phi \operatorname{T}_{x_0} \operatorname{M}_{\omega_0} \calF_{-\alpha} f_{\pm} (x,\omega) } = \abs{ \calV_\phi \calF_{-\alpha} f_{\pm} (x-x_0,\omega-\omega_0) } \\
        &= \abs{ \calV_{\calF_{-\alpha} \phi} \calF_{-\alpha} f_{\pm} (x-x_0,\omega-\omega_0) }
        = \abs{ \calV_{\phi} f_{\pm} \left(R_{-\alpha}(x-x_0,\omega-\omega_0)\right) },
    \end{aligned}
    \end{equation}
    for $x,\omega \in \bbR$, where we have used the covariance property of the Gabor transform (see Lemma 3.1.3 on p.~41 of \cite{grochenig2001foundations}) in the second step, the invariance of the Gaussian under the fractional Fourier transform in the third step and Lemma \ref{lem:frftandgabor} in the fourth and final step. It follows immediately from the above consideration and Lemma \ref{lem:gaborfpm} that 
    \begin{equation*}
        \abs{\calV_\phi f_+^{\alpha,\lambda}} = \abs{\calV_\phi f_-^{\alpha,\lambda}} \mbox{ on } R_\alpha(a \bbZ \times \bbR) + \lambda.
    \end{equation*}
    
    Additionally, we may remember the proof of Lemma \ref{lem:gaborfpm} and estimate that 
    \begin{align*}
        \abs{\calV_\phi f_{\pm}(x,\omega) } &= \frac{\rme^{-\frac{\pi}{8a^2}}}{2 \sqrt{2}} \abs{ ( 1 \mp \rmi ) \rme^{\frac{\pi}{2a}\left( \omega + \rmi x \right)} + ( 1 \pm \rmi ) \cdot \rme^{-\frac{\pi}{2a}\left( \omega + \rmi x \right)} } \rme^{-\frac{\pi}{2}\left(x^2 + \omega^2\right)} \\
        &\leq \rme^{-\frac{\pi}{8a^2}} \rme^{\frac{\pi}{2a} \abs{\omega}} \cdot \rme^{-\frac{\pi}{2}\left(x^2 + \omega^2\right)},
    \end{align*}
    for $x,\omega \in \bbR$. It follows that 
    \begin{align*}
        \abs{\calV_\phi f_{\pm}^{\alpha,\lambda}(x,\omega) } &= \abs{ \calV_{\phi} f_{\pm} \left(R_{-\alpha}(x-x_0,\omega-\omega_0)\right) } \\
        &\leq \rme^{-\frac{\pi}{8a^2}} \rme^{\frac{\pi}{2a} \abs{(x-x_0) \sin \alpha + (\omega-\omega_0) \cos \alpha}} \cdot \rme^{-\frac{\pi}{2}\left((x-x_0)^2 + (\omega-\omega_0)^2\right)} \\
        &\leq \rme^{-\frac{\pi}{8a^2}} \rme^{\frac{\pi}{2a} \left( \abs{x-x_0} + \abs{\omega-\omega_0} \right) } \cdot \rme^{-\frac{\pi}{2}\left((x-x_0)^2 + (\omega-\omega_0)^2\right)}.
    \end{align*}
    Let us now consider $\epsilon > 0$ arbitrary but fixed. It is then readily seen that there exists $C_\epsilon \geq 1$ depending on $\epsilon$ as well as $a$, $x_0$ and $\omega_0$ such that 
    \begin{equation*}
        \abs{\calV_\phi f_{\pm}^{\alpha,\lambda}(x,\omega)} \leq C_\epsilon \cdot \rme^{-\left( \frac{\pi}{2} - \epsilon \right) \left( x^2 + \omega^2 \right) }, \qquad x,\omega \in \bbR.
    \end{equation*}
\end{proof}

\begin{remark}
    The reader might note that we could have directly constructed the functions $f_\pm^{\alpha,\lambda}$ using the Hadamard factorisation theorem as well as the relation between the Gabor and the Bargmann transform. This is indeed how we first constructed the examples $f_\pm^{\alpha,\lambda}$. The formulation involving the fractional Fourier transform was suggested by the first reviewer and it has allowed for a much shorter and more elegant presentation of the paper.
\end{remark}

\paragraph{Acknowledgements} The authors want to especially thank the first reviewer for their comments which have allowed us to simplify the proofs and thereby improved the overall presentation of the present paper greatly. Furthermore, the authors acknowledge funding through SNF Grant 200021\_184698.

\bibliographystyle{plain}
\bibliography{sources} 

\end{document}